\date{September 8, 2009} 

\documentclass[10pt]{article}
\usepackage{amssymb}
\usepackage{amsmath}

\font\tengoth=eufm10 at 10pt
\font\sevengoth=eufm7 at 6pt
\newfam\gothfam
\textfont\gothfam=\tengoth
\scriptfont\gothfam=\sevengoth

\newcommand{\g}{{\mathfrak g}}

\newcommand{\z}{{\mathfrak z}}

\newcommand{\fb}{{\mathfrak b}}

\newcommand{\fh}{{\mathfrak h}}

\newcommand{\fl}{{\mathfrak l}}

\newcommand{\fn}{{\mathfrak n}}

\newcommand{\fp}{{\mathfrak p}}

\newcommand{\fs}{{\mathfrak s}}

\newcommand{\fu}{{\mathfrak u}}

\renewcommand{\:}{\colon}
\newcommand{\1}{\mathbf{1}}
\newcommand{\0}{{\bf 0}}

\newcommand{\cA}{\mathcal{A}}
\newcommand{\cB}{\mathcal{B}}

\newcommand{\cL}{\mathcal{L}}

\newcommand{\cO}{\mathcal{O}}
\newcommand{\cH}{\mathcal{H}}

\renewcommand{\phi}{\varphi}
\newcommand{\dd}{{\tt d}}

\newcommand{\trile}{\trianglelefteq}
\newcommand{\subeq}{\subseteq}

\newcommand{\into}{\hookrightarrow}
\newcommand{\eps}{\varepsilon}

\def\onto{\to\mskip-14mu\to}

\newcommand{\N}{{\mathbb N}}
\newcommand{\Z}{{\mathbb Z}}

\newcommand{\C}{{\mathbb C}}

\newcommand{\K}{{\mathbb K}}

\newcommand{\bS}{{\mathbb S}}

\renewcommand{\hat}{\widehat}

\renewcommand{\tilde}{\widetilde}

\renewcommand{\L}{\mathop{\bf L{}}\nolimits}

\newcommand{\GL}{\mathop{{\rm GL}}\nolimits}

\newcommand{\SL}{\mathop{{\rm SL}}\nolimits}

\newcommand{\SO}{\mathop{{\rm SO}}\nolimits}

\newcommand{\U}{\mathop{\rm U{}}\nolimits}

\newcommand{\Sp}{\mathop{{\rm Sp}}\nolimits}

\newcommand{\gl}  {\mathop{{\mathfrak{gl} }}\nolimits}

\newcommand{\fsl} {\mathop{{\mathfrak{sl} }}\nolimits}

\newcommand{\su}  {\mathop{{\mathfrak{su} }}\nolimits}

\newcommand{\Ad}{\mathop{{\rm Ad}}\nolimits}

\newcommand{\Aut}{\mathop{{\rm Aut}}\nolimits}
\newcommand{\Der}{\mathop{{\rm Der}}\nolimits}

\newcommand{\id}{\mathop{{\rm id}}\nolimits}

\renewcommand{\dim}{\mathop{{\rm dim}}\nolimits}

\newcommand{\Spann}{\mathop{{\rm span}}\nolimits}
\newcommand{\ev}{\mathop{{\rm ev}}\nolimits}

\newcommand{\oline}{\overline}

\newcommand{\res}{\vert}

\newcommand{\ssssarr}{\hbox to 15pt{\rightarrowfill}}
\newcommand{\sssarr}{\hbox to 20pt{\rightarrowfill}}
\newcommand{\ssarr}{\hbox to 30pt{\rightarrowfill}}
\newcommand{\sarr}{\hbox to 40pt{\rightarrowfill}}
\newcommand{\arr}{\hbox to 60pt{\rightarrowfill}}
\newcommand{\larr}{\hbox to 60pt{\leftarrowfill}}
\newcommand{\Arr}{\hbox to 80pt{\rightarrowfill}}

\def\theoremname{Theorem}
\def\propositionname{Proposition}
\def\corollaryname{Corollary}
\def\lemmaname{Lemma}
\def\remarkname{Remark}
\def\conjecturename{Conjecture} 

\def\definitionname{Definition}
\def\exercisename{Exercise}
\def\examplename{Example}
\def\examplesname{Examples}
\def\problemname{Problem}
\def\problemsname{Problems}
\def\proofname{Proof}

\def\satzname{Satz} 
\def\koroname{Korollar}
\def\folgname{Folgerung}
\def\bemerkname{Bemerkung}
\def\aufgname{Aufgabe}

\def\beisname{Beispiel}
\def\beissname{Beispiele}
\def\bewname{Beweis}

\def\@thmcounter#1{\noexpand\arabic{#1}}
\def\@thmcountersep{}
\def\@begintheorem#1#2{\it \trivlist \item[\hskip 
\labelsep{\bf #1\ #2.\quad}]}
\def\@opargbegintheorem#1#2#3{\it \trivlist
      \item[\hskip \labelsep{\bf #1\ #2.\quad{\rm #3}}]}
\makeatother
\newtheorem{theor}{\theoremname}[section]
\newtheorem{propo}[theor]{\propositionname}
\newtheorem{coro}[theor]{\corollaryname}
\newtheorem{lemm}[theor]{\lemmaname}

\newenvironment{thm}{\begin{theor}\it}{\end{theor}}
\newenvironment{theorem}{\begin{theor}\it}{\end{theor}}

\newenvironment{prop}{\begin{propo}\it}{\end{propo}}
\newenvironment{proposition}{\begin{propo}\it}{\end{propo}}

\newenvironment{lem}{\begin{lemm}\it}{\end{lemm}}
\newenvironment{lemma}{\begin{lemm}\it}{\end{lemm}}

\newtheorem{rema}[theor]{\remarkname}

\newenvironment{remark}{\begin{rema}\rm}{\end{rema}}
\newenvironment{rem}{\begin{rema}\rm}{\end{rema}}

\newtheorem{stepnow}[theor]{}

\newtheorem{defin}[theor]{\definitionname} 

\newenvironment{defn}{\begin{defin}\rm}{\end{defin}}

\newtheorem{exerc}[theor]{\exercisename}

\newtheorem{exa}[theor]{\examplename}

\newtheorem{exas}[theor]{\examplesname}

\newenvironment{exs}{\begin{exas}\rm}{\end{exas}}

\newtheorem{conj}[theor]{\conjecturename}

\newtheorem{pro}[theor]{\problemname}

\newtheorem{prs}[theor]{\problemsname}

\newcommand{\qed}{{\unskip\nobreak\hfil\penalty50\hskip .001pt \hbox{}
          \nobreak\hfil
          \vrule height 1.2ex width 1.1ex depth -.1ex
           \parfillskip=0pt\finalhyphendemerits=0\medbreak}\rm}

{\hfill\qed\end{trivlist}}
\newenvironment{proof}{\begin{trivlist}\item[\hskip%
\labelsep{\bf\proofname.\quad}]}%
{\hfill\qed\end{trivlist}}
\newenvironment{prf}{\begin{trivlist}\item[\hskip%
\labelsep{\bf\proofname.\quad}]}%
{\hfill\qed\end{trivlist}}

\newenvironment{Proof*}{\begin{trivlist}\item[\hskip%
\labelsep{\bf\proofname.\quad}]}%
{\end{trivlist}}

\newcommand{\pmat}[1]{\begin{pmatrix} #1 \end{pmatrix}}

{\hfill\qed\end{trivlist}}

\newenvironment{beweis*}{\begin{trivlist}\item[\hskip%
\labelsep{\bf\bewname.\quad}]}%
{\end{trivlist}}

\newtheorem{satzn}[theor]{\satzname}

\newtheorem{koro}[theor]{\koroname}

\newtheorem{folg}[theor]{\folgname}

\newtheorem{bem}[theor]{\bemerkname}

\newtheorem{aufg}[theor]{\aufgname}

\newtheorem{aufgn}[theor]{\aufgname}

\newtheorem{beis}[theor]{\beisname}

\newtheorem{beiss}[theor]{\beissname}

\usepackage{xypic}

\renewcommand{\1}{\mathbf{1}}
\renewcommand{\cL}{\mathcal{L}}
\newcommand{\cI}{\mathcal{I}}
\renewcommand{\cH}{\mathcal{H}}
\newcommand{\p}{\mathfrak{p}}

\begin{document} 

\title{Borel--Weil Theory for Groups over Commutative 
Banach Algebras} 

\author
{Karl-Hermann
Neeb\footnote{Technische Universit\"at Darmstadt,
Schlossgartenstrasse 7, D-64289 Darmstadt, Germany, 
neeb@mathematik.tu-darmstadt.de},
Henrik Sepp\"anen\footnote{Technische Universit\"at Darmstadt,
Schlossgartenstrasse 7, D-64289 Darmstadt, Germany, 
seppaenen@mathematik.tu-darmstadt.de. The second author was supported
by a post-doctoral fellowship from the Swedish Research Council.}}

%

\maketitle

\begin{abstract} Let $\cA$ be a commutative unital 
Banach algebra, $\g$ be a semisimple complex Lie algebra 
and $G(\cA)$ be the $1$-connected Banach--Lie group 
with Lie algebra $\g \otimes \cA$. 
Then there is a natural concept of a parabolic subgroup 
$P(\cA)$ of $G(\cA)$ and we obtain 
generalizations $X(\cA) := G(\cA)/P(\cA)$ of the generalized flag manifolds. 
In this note we provide an explicit description of all 
homogeneous holomorphic line bundles over $X(\cA)$ 
with non-zero holomorphic sections. In particular, we show that 
all these line bundles are tensor products of pullbacks 
of line bundles over $X(\C)$ by evaluation maps. 

For the special case where $\cA$ is a $C^*$-algebra, our 
results lead to a complete classification of all irreducible 
involutive holomorphic representations of $G(\cA)$ on Hilbert spaces. 
\end{abstract}

\section{Introduction} \label{sec:1} 

If $\g$ is a finite dimensional complex semisimple Lie algebra 
and $\cA$ is a unital commutative Banach algebra, then 
$\g(\cA) := \g \otimes \cA$ carries a natural Banach--Lie algebra 
structure with respect to the $\cA$-bilinear extension of the bracket. 
As we shall see below, there always exists a 
($1$-connected) Banach--Lie group $G(\cA)$ with Lie algebra 
$\g(\cA)$. For any parabolic subalgebra $\fp \subeq \g$,  
we then obtain a connected Banach--Lie subgroup 
$P(\cA)$ with Lie algebra $\fp(\cA) := \fp \otimes \cA \subeq \g(\cA)$, 
which leads to the complex homogeneous spaces $X(\cA) := G(\cA)/P(\cA)$ 
generalizing the finite dimensional complex flag manifolds. 

In \cite{MNS09} we have studied homogeneous vector bundles 
over a class of Banach manifolds generalizing 
those of the form $G(\cA)/P(\cA)$. 
Some of the main results of that paper are 
that for each holomorphic Banach representation $\rho \: P(\cA) \to \GL(E)$,  
the space of holomorphic sections of the associated bundle 
$G \times_\rho E$ always carries a natural Banach space structure 
turning it into a holomorphic $G(\cA)$-module and that every 
irreducible holomorphic $G(\cA)$-module embeds in such a space 
of holomorphic sections. 
These results constitute natural extensions of 
Borel--Weil theory for finite dimensional reductive complex Lie groups. 

In this paper we obtain a complete classification of all 
homogeneous holomorphic line bundles over $X(\cA)$ 
with non-zero holomorphic sections. In particular, we show that 
all these line bundles are tensor products of pullbacks 
of line bundles over the finite dimensional compact complex 
manifold $X(\C)$ by evaluation maps  
$\phi^X_\eta \: X(\cA) \to X(\C)$ induced by unital algebra 
homomorphisms $\eta \: \cA \to \C$. 

If, in addition, $\cA$ is a $C^*$-algebra, then the involution 
of $\cA$ and the Cartan involution on $\g$ can be combined to 
an involution on the Banach--Lie algebra $\g(\cA)$, which leads to an 
antiholomorphic involution $*$ on the corresponding Lie group 
$G(\cA)$. For these groups the natural class of representations 
are those holomorphic representations 
$\pi \: G(\cA) \to \GL(\cH)$ on complex Hilbert spaces $\cH$ 
which are compatible with the involution. On the unitary groups 
$$ \U(G(\cA)) := \{ g \in G(\cA) \: g^* = g^{-1}\} $$
they restrict to norm continuous unitary representation, 
from which they can be reconstructed by analytic extension 
(cf.\ \cite{Ne98}). We show that all irreducible representations 
of this kind can be realized in holomorphic homogeneous 
line bundles over $X(\cA)$, from which we then derive a 
complete classification. Surprisingly, it turns out that all 
the irreducible representations are actually finite dimensional 
and factor through multi-evaluation maps $G(\cA) \to G(\C^N) \cong G(\C)^N$. 

The structure of the present paper is as follows. 
In Section~\ref{sec:2} we explain our setup and collect some 
structural information on the Lie algebras $\g(\cA)$. 
Section~\ref{sec:3} is completely independent of our representation 
theoretic framework. Its main result is Theorem~\ref{holchar}, 
asserting that every multiplicative holomorphic map 
$\chi \: \cA \to \C$ on a unital commutative Banach algebra 
is a finite product of algebra homomorphisms. This observation 
is the key to our results on the characterization 
of the holomorphic line bundles $\cL_\chi \to X(\cA)$ 
associated to holomorphic characters $\chi \: P(\cA) \to \C^\times$. 
Theorem~\ref{big} provides a characterization of those characters 
$\chi$ for which $\cL_\chi$ has non-zero holomorphic sections,  
as those which are products of characters pulled back from 
dominant characters of $P(\C)$ by homomorphisms 
$\phi_\eta \: P(\cA) \to P(\C)$ induced by algebra homomorphisms 
$\eta \: \cA \to \C$. 
This theorem is proved in Section~\ref{sec:5} by showing first
that, for the special case of $\g = \fsl_2(\C)$, it follows from 
Theorem~\ref{holchar} and then deriving the general case by 
applying the $\fsl_2$-case to $\fsl_2$-subalgebras corresponding 
to simple roots of $\g$. In Section~\ref{sec:6} we 
apply all that to the special case where $\cA$ is a unital 
$C^*$-algebra. Finally, we show in Section~\ref{sec:7} that, 
in general, the space $\cO_\chi(G(\cA))$ of holomorphic sections  
of $\cL_\chi$ is not finite-dimensional, 
although the corresponding line bundle is a pullback of a finite dimensional 
one.

For $\cA = C^k(\bS^1)$, $k \in \N_0$, the groups 
$G(\cA)$ are variants of complex loop groups. It is well-known that, 
at least for $k = \infty$, these groups have interesting central extensions 
with a very rich representation theory (by unbounded operators) 
(cf.\ \cite{PS86}, \cite{Ne01}), 
so that our results can also be understood 
as a contribution to the description of those representations 
of the centrally extended groups which are trivial on the center. 

As a consequence of our main result, 
the space of holomorphic 
sections of $\cL_\chi$ contains a finite dimensional 
$G(\cA)$-invariant subspace whenever it is non-zero. In particular, 
this leads to a natural class of finite dimensional representations 
of groups of the type $G(\cA)$ that deserve the name 
{\it evaluation representations}. Finite dimensional 
representations of Lie algebras of the form 
$\g \otimes \cA$, $\cA$ a unital commutative algebra, 
are presently under active investigation from an algebraic 
point of view. 
In \cite{Se09} one finds a survey on this theory for the 
case where $\cA$ is an algebra of Laurent polynomials. 
For the larger class of Lie algebras of the form 
$(\g \otimes \cA)^\Gamma$, where 
$\cA$ is the algebra of regular functions on an affine variety 
and $\Gamma$ a finite group acting on $\g$ and $\cA$, the irreducible 
finite dimensional representations have recently been 
classified by Neher, Savage and Senesi (\cite{NSS09}). 

Also closely related to our setting is the notion 
of a Weyl module introduced in \cite{CP01}. 
These are the maximal finite dimensional 
modules of algebras of the form $\g \otimes \cA$ 
generated by eigenvectors of $\fb \otimes \cA$, where 
$\fb$ is a Borel subalgebra of $\g$. The connection 
to our context is as follows. For any line bundle 
$\cL_\chi$ as above, we can identify its space of 
holomorphic sections with a certain space $\cO_\chi(G(\cA))$ 
of holomorphic functions on $G(\cA)$, so that the evaluation 
$\ev_\1 \: \cO_\chi(G(\cA)) \to \C$ is a morphism of 
$P(\cA)$-modules. In particular, $\ev_\1$ can be viewed as a 
$P(\cA)$-eigenvector in the 
dual $G(\cA)$-module $\cO_\chi(G(\cA))^*$. 
If $\cO_\chi(G(\cA))$ is finite dimensional (which is in 
particular the case if $\cA$ is finite dimensional; 
cf.~\cite[Cor.~3.9]{MNS09}), then 
$\cO_\chi(G(\cA))^*$ is a finite dimensional $G(\cA)$-module 
generated by a $P(\cA)$-eigenvector. If the parabolic 
$P(\cA)$ is minimal, these dual modules are Weyl modules. 
Conversely, the description of $\cO_\chi(G(\cA))$ as 
coinduced modules on the Lie algebra 
level in \cite[Sec.~2]{MNS09} implies that, whenever one can 
translate between the analytic and the algebraic setting, 
Weyl modules can be realized as duals of finite dimensional 
$G(\cA)$-invariant spaces of holomorphic sections of some line 
bundle $\cL_\chi$. For recent results on the structure of 
Weyl modules we refer to \cite{FoLi07}, \cite{FL04}. 

\section{Preliminaries} \label{sec:2}

Let $\g$ be a finite dimensional complex semisimple Lie algebra, 
$\fh \subeq \g$ be a Cartan subalgebra, and $\Delta \subeq \fh^*$ be 
the corresponding root system, so that we have the root decomposition 
\begin{eqnarray*}
\g = \mathfrak{h} \oplus \bigoplus_{\alpha \in \Delta} \g_{\alpha}.
\end{eqnarray*}
We write $\check \alpha \in \fh$ for the coroot associated to 
$\alpha \in \Delta$, i.e., the unique element $\check \alpha
 \in [\g_\alpha, \g_{-\alpha}]$ with $\alpha(\check \alpha) = 2$. 
Fix a positive system $\Delta^+$, and let
$\Pi = \{ \alpha_1,\ldots, \alpha_r\}$ denote the corresponding 
simple roots. 

In the following $\cA$ always denotes a complex unital commutative 
Banach algebra. Then 
$\g(\mathcal{A}):=\g \otimes \mathcal{A}$, equipped with Lie bracket
defined by \begin{eqnarray*}
\left[x_1 \otimes a_1, x_2 \otimes _2\right]:=[x_1,x_2] \otimes a_1a_2
\end{eqnarray*}
is a Banach--Lie algebra with respect to the natural tensor product 
topology, for which $\g(\cA) \cong \cA^{\dim \g}$ as a Banach space. 
We consider $\g(\cA)$ as a Lie algebra over the algebra $\cA$, 
hence sometimes write $x \otimes a \in \g \otimes \cA$ also as $ax$. 
From the $\fh$-weight space decomposition 
\begin{eqnarray*}
\g(\mathcal{A})=(\mathfrak{h} \otimes \mathcal{A}) \oplus 
\bigoplus_{\alpha \in \Delta} (\g_{\alpha} \otimes \mathcal{A}), 
\end{eqnarray*}
we derive that $\g(\cA)$ is weakly $\Delta$-graded 
in the sense of \cite[Def.~1.1]{MNS09} because 
it contains $\g_\Delta := \g \otimes \1$ and we have 
the $\fh$-weight decomposition from above. 

To each subset $\Pi_\Sigma \subeq \Pi$, we associate a parabolic system of 
roots, defined by 
\begin{equation}
  \label{eq:parsyst}
\Sigma := (-\Delta^+) \cup (\Delta \cap 
\Spann_\Z (\Pi\setminus \Pi_\Sigma)). 
\end{equation}
If $x_\Sigma \in \fh$ is  such that 
$$ \alpha_i(x_\Sigma) = \begin{cases} 
0 & \text{for } \alpha_i \not\in \Pi_\Sigma \\ 
-1 & \text{for } \alpha_i \in \Pi_\Sigma,
\end{cases} $$
then 
$\Sigma=\{ \alpha \in \Delta \mid \alpha(x_\Sigma) \geq 0\}$.
Let $\fp:=\fp(\C) := 
\mathfrak{h} \oplus \bigoplus_{\alpha \in \Sigma} \g_{\alpha}$ 
be the parabolic subalgebra corresponding to $\Sigma$.
Then $\fp(\mathcal{A}):=\fp \otimes \mathcal{A}$ is called a 
{\it parabolic subalgebra} of
$\g(\mathcal{A})$.

The Lie-algebra $\g(\mathcal{A})$ integrates to a Banach--Lie group. 
In fact, if we choose some faithful representation $\g 
\rightarrow \mathfrak{gl}_n(\C)$, 
$\g(\mathcal{A})$ is a closed subalgebra of the Banach--Lie algebra $\mathfrak{gl}_n(\cA)$ of
$n \times n$-matrices with entries in $\mathcal{A}$. This Lie algebra integrates
to the Lie group $\GL_n(\mathcal{A})$ of all invertible matrices with entries in $\mathcal{A}$.
Hence $\g(\mathcal{A})$ is a closed Lie subalgebra of a Lie algebra 
of a linear Lie group, and therefore integrates to a Lie group 
(\cite{Mais62}). 
Let  $G(\mathcal{A})$, resp., $G(\C)$ be simply connected Banach--Lie groups
with Lie algebras $\g(\mathcal{A})$, resp., $\g$, and define 
Lie subgroups $P(\mathcal{A})$, resp., 
$P(\C)$ as the connected subgroups with
Lie algebras $\fp(\mathcal{A})$, resp., $\fp(\C) = \fp$. 

\begin{remark}\label{r:cparabolic}
(a) The Lie algebra $\p$ is a semidirect sum 
$\fp = \fu \rtimes \fl$, where 
$$ \mathfrak{l}=\mathfrak{h} \oplus \bigoplus_{\alpha(x_\Sigma)=0} \g_{\alpha} 
\quad \mbox{ and } \quad \mathfrak{u}=\bigoplus_{\alpha(x_\Sigma)>0} \g_{\alpha}.$$
Moreover, the subalgebra $\mathfrak{l}$ is a semidirect sum 
$\mathfrak{l}=\mathfrak{c} \ltimes \mathfrak{s}$, where
$$ \mathfrak{c}:=\Spann \check\Pi_\Sigma 
\quad \mbox{ and } \quad 
\mathfrak{s}:=[\fl, \fl] 
= \Big(\bigoplus_{\alpha(x_\Sigma)=0}\C \check{\alpha}\Big)
 \oplus  \bigoplus_{\alpha(x_\Sigma)=0}\g_{\alpha}.$$
Let $U:=\exp \mathfrak{u}$ and  
$L:=N_{P(\C)}(\mathfrak{l}) \cap N_{P(\C)}(\fu)$. 
Then the multiplication map $L \ltimes U \rightarrow P(\C)$, 
$(l,u) \mapsto lu$, is a holomorphic isomorphism. In particular, 
$L$ is connected because $P(\C)$ is connected. 

(b) We define the groups $C$ and $S$ as the integral subgroups
of $L$ with Lie algebras $\mathfrak{c}$ and $\mathfrak{s}$, respectively.
Let $H_S$ be the integral subgroup of $S$ with Lie algebra
$\fh_S := \fh \cap \fs$. Then the integral subgroup $H$ 
with Lie algebra $\fh$ satisfies 
$$H \cong \fh/\ker(\exp\res_\fh) 
= \fh/2\pi i \Z[\check \Delta] \cong (\C^\times)^r, $$
and this implies that the multiplication map 
$C \times H_S \to H$ is an isomorphism of abelian complex 
Lie groups. The product group $C\cdot S \subeq G$, 
being the integral subgroup with Lie algebra $\mathfrak{c} \oplus
\mathfrak{s}$, equals $L$. Hence the multiplication map
$C \times S \rightarrow L$ is surjective. 
To see that it is also injective, we note that its kernel is discrete
and normal, hence contained in the center 
$Z(C \times S) \subseteq C \times H_S$. The injectivity
now follows since the restriction of the multiplication to $C \times H_S$ 
is injective.
\end{remark}

Let $\mathfrak{n}(\mathcal{A}):=\sum_{\alpha \in \Delta \setminus
\Sigma} \mathfrak{g}_{\alpha} \otimes \mathcal{A},$ and let
$N(\mathcal{A}):=\exp(\mathfrak{n}(\cA))$ be the corresponding
integral subgroup of $G(\mathcal{A})$.
Recall from \cite[Prop. 1.11]{MNS09}  that the multiplication map
\begin{equation}\label{NPfact}
N(\mathcal{A}) \times P(\mathcal{A}) 
 \rightarrow  G(\mathcal{A}), \quad (n,p) \mapsto np 
\end{equation}
is biholomorphic onto an open subset. 
It follows that $P(\mathcal{A})$ is a complemented
Lie subgroup, so that the quotient space $X(\cA) := 
G(\mathcal{A})/P(\mathcal{A})$ is
a complex Banach manifold and the projection map 
$\pi_\cA: G(\mathcal{A}) \rightarrow X(\mathcal{A})$ is a holomorphic 
submersion defining a holomorphic $P(\cA)$-principal bundle. 

\section{Multiplicative holomorphic functions} \label{sec:3}

In this section we are concerned with 
holomorphic functions 
$\phi \: \cA \to \C$ on a commutative unital Banach algebra 
which are multiplicative, i.e., 
$\phi(ab) = \phi(a)\phi(b)$ for $a, b \in \cA$. 
Clearly, every algebra homomorphism has this property, 
and so does every finite product of algebra homomorphisms. 
The main result of this section (Theorem~\ref{holchar}) 
asserts the converse, 
namely that any such $\phi$ is a finite product of algebra 
homomorphisms. 

We start with a simple algebraic observation. 

\begin{lem}\label{lem:idealext}
Let $\Gamma$ be a finite group, and let 
$\sigma: \Gamma \rightarrow \Aut(\cA)$
be a representation of $\Gamma$ as automorphisms of 
a unital algebra $\cA$ over a field of characteristic zero. 
Let $\cA^{\Gamma}$ denote the
subalgebra of  $\Gamma$-invariants. Then every proper left ideal $I$ in 
$\cA^{\Gamma}$ generates a proper left ideal in $\cA$.
\end{lem} 

\begin{proof}
Since the representation $\sigma$ of $\Gamma$ on $\cA$ is locally finite, we can
decompose $\cA$ as a finite direct sum
$\cA=\bigoplus_{\tau \in \hat{\Gamma}}\cA_{\tau},$ 
where $\widehat{\Gamma}$ denotes the set of irreducible representations of 
$\Gamma$, and $\cA_{\tau}$ is the $\tau$-isotypic component, i.e., 
the sum of all irreducible subrepresentations of $\sigma$ which are
equivalent to $\tau$. Observe that $\cA^{\Gamma}$ is the isotypic component
of the trivial representation.
For any $a_0 \in \cA^{\Gamma}$, the map $x \mapsto a_0x$ is $\Gamma$-equivariant,
so that $\cA^{\Gamma} \cA_{\tau} \subseteq \cA_{\tau}$
for any $\tau \in \hat{\Gamma}$.
Assume now that the left ideal $\cA I$ generated by $I$ is not proper, i.e, that
$\1 \in \cA 
I=\bigoplus_{\tau}\cA_{\tau}I$. Since $\1 \in \cA^{\Gamma}$, it follows
that $\1 \in \cA^{\Gamma}I \subseteq I$, which contradicts that $I$ is a proper ideal.
Hence $\cA I$ is a proper ideal.
\end{proof}

\begin{prop}\label{extendhom}
If $\cA$ is a unital commutative complex Banach algebra 
and $\Gamma \subeq \Aut(\cA)$ a finite subgroup, then any algebra homomorphism 
$\phi \: \cA^{\Gamma} \rightarrow \C$ extends to 
an algebra homomorphism $\tilde\phi: \cA \rightarrow \C$ and any such 
homomorphism is continuous.
\end{prop}

\begin{prf} The kernel $I := \ker \phi$ is a proper ideal in the 
subalgebra $\cA^\Gamma$, so that the preceding lemma implies that 
$I$ is contained in a proper ideal of $\cA$. 
In particular, it is contained in a maximal ideal of $\cA$, so that 
the Gelfand--Mazur Theorem implies the existence 
of a (continuous) homomorphism $\tilde\phi \: \cA \to \C$ with 
$\ker \tilde\phi \cap \cA^\Gamma = \ker \phi$ (\cite[Thm.~11.5]{Ru91}). 
Since $\cA^\Gamma = I \oplus \C \1$, it follows that 
$\tilde\phi$ extends $\phi$. 
Its continuity follows from \cite[Thm.~11.10]{Ru91}. 
\end{prf}

\begin{theorem}\label{holchar}
Let $\cA$ be a unital commutative Banach algebra 
and $\varphi: \cA \rightarrow \C$ be a
holomorphic character of the multiplicative semigroup $(\cA, \cdot)$.
Then there exist finitely many continuous algebra homomorphisms 
$\chi_1, \ldots, \chi_n \: \cA \to \C$ such that
\begin{equation}
\varphi=\chi_1 \cdots \chi_n.
\end{equation}
\end{theorem}

\begin{proof} We first claim that $\varphi$ is a 
homogeneous polynomial of some degree $n$, i.e., 
there exists a symmetric $n$-linear map 
$\tilde{\varphi}: \cA^n \rightarrow \C$ with 
$\phi(a) = \tilde\phi(a,\ldots,a)$ for every $a \in \cA$. 
Indeed, the map $\C \rightarrow \C, z \mapsto \varphi(z\1),$ is 
holomorphic and multiplicative, hence of the form $z \mapsto z^n$ 
for some $n \in \N_0$. From this, we get the homogeneity condition
\begin{eqnarray}
\varphi(za)=z^n\varphi(a) \quad \mbox{ for } \quad 
z \in \C, a \in \cA. \label{homo}
\end{eqnarray}
On the other hand, we have a power series expansion
\begin{eqnarray}
\varphi=\sum_{k=0}^{\infty}\varphi_k \label{taylor}
\end{eqnarray}
at the origin,
where $\varphi_k: \cA \rightarrow \C$ is a homogeneous 
polynomial of degree $k$.
Comparison of \eqref{homo} and \eqref{taylor} yields $\varphi=\varphi_n$.
We can thus write $\varphi(a)=\tilde{\varphi}(a, \ldots, a)$ 
for a continuous $n$-linear map $\tilde{\varphi}: \cA^n \rightarrow \C$.

Now let $\cA^{\otimes n}$ denote the projective $n$-fold 
tensor product of $\cA$, which is the completion of the algebraic 
tensor product with respect to the maximal cross norm. 
It has the universal property that continuous linear maps 
$\cA^{\otimes n} \to X$ to a Banach space are in one-to-one correspondence 
with continuous $n$-linear maps $\cA^n \to X$. 
From the universal property and the associativity 
of projective tensor products it  
easily follows that $\cA^{\otimes n}$ carries a natural 
unital commutative Banach algebra structure, determined by 
\begin{eqnarray*}
(a_1 \otimes \cdots \otimes a_n)(b_1 \otimes \cdots \otimes b_n)
:=a_1b_1 \otimes \cdots \otimes a_nb_n. 
\end{eqnarray*} 
The symmetric group $S_n$ acts by automorphisms on $\cA^{\otimes n}$ 
by the continuous linear extensions of the maps which permute
the tensor factors, i.e.
\begin{eqnarray*}
\sigma(a_1 \otimes \cdots \otimes a_n):=a_{\sigma^{-1}(1)} \otimes \cdots 
\otimes a_{\sigma^{-1}(n)}, \quad \sigma \in S_n.
\end{eqnarray*}
The fixed point algebra $S^n(\cA) := (\cA^{\otimes n})^{S_n}$ 
also is a unital Banach algebra. It 
is topologically generated by tensors of the form  
\begin{eqnarray*}
a_1 \vee \cdots \vee a_n:=\frac{1}{n!}
\sum_{\sigma \in S_n}a_{\sigma(1)} \otimes \cdots \otimes a_{\sigma(n)},
\end{eqnarray*}
and by polarization it is actually generated by the diagonal elements
$a \vee \cdots \vee a=a \otimes \cdots \otimes a$.

With the universal property of the Banach space $S^n(\cA)$, 
we find a continuous linear map $\psi \: S^n(\cA) \to \C$ 
with $\tilde\phi(a,\ldots, a) = \psi(a \otimes \cdots \otimes a)$ 
for $a \in \cA$. For the diagonal generators of $S^n(\cA)$ we now have 
\begin{eqnarray*}
\psi((a \otimes \cdots \otimes a)(b \otimes \cdots \otimes b))
&=&\psi(ab \otimes \cdots \otimes ab)=\phi(ab)=\phi(a)\phi(b)\\
&=&\psi(a \otimes \cdots \otimes a)\psi(b \otimes \cdots \otimes b).
\end{eqnarray*}
From the linearity of $\psi$ and its multiplicativity 
on a set of topological  linear generators of the Banach algebra 
$S^n(\cA)$, it now follows that $\psi$ is an algebra homomorphism. 
By Proposition~\ref{extendhom}, we can extend $\psi$ 
to an algebra homomorphism $\chi: \cA^{\otimes n} \rightarrow \C$.
Then
\begin{eqnarray*}
\phi(a)&=&\chi(a \otimes \cdots \otimes a)\\
&=&\chi(a \otimes \1 \otimes \cdots \1)\chi(\1 \otimes a \otimes \1 \otimes 
\cdots \otimes \1)
\cdots \chi(\1 \otimes \cdots \otimes \1 \otimes a)\\
&=&\chi_1(a)\cdots \chi_n(a), 
\end{eqnarray*}
where $\chi_i \: \cA \rightarrow \C$ is the character
\begin{eqnarray*}
\chi_i(a):=
\chi(\1 \otimes \cdots \otimes \1 \otimes a \otimes \1 \otimes \cdots 
\otimes \1)
\end{eqnarray*}
with $a$ occurring as the $i$th factor. 
Since every character of $\cA$ is automatically continuous 
(\cite[Thm.~11.10]{Ru91}), this proves the theorem.
\end{proof}

\begin{rem} The preceding theorem basically asserts that for the 
commutative Banach algebra $S^n(\cA)$, the natural map 
$$ (\cA^{\otimes n})\,\hat{} \cong \hat\cA^n \to  S^n(\cA)\,\hat{} $$
is surjective, 
and it is not hard to see that this leads to a topological isomorphism 
$$ S^n(\cA)\,\hat{} \cong \hat\cA^n/S_n, $$
where the symmetric group $S_n$ acts on $\hat\cA^n$ by permutations. 

It is an interesting problem to find a non-commutative analog of this 
result (for $C^*$-algebras) (cf.\ \cite{Ar87} for some results 
pointing in this direction). 
\end{rem}

\begin{rem} There is an interesting algebraic version of the 
preceding theorem which may also be of interest in other contexts. 
Let $\cA$ be a finitely generated unital commutative 
algebra over an algebraically closed field $\K$ of characteristic zero and 
$\phi \: \cA \to \K$ a non-zero multiplicative polynomial map. 
Then there exists an $n \in \N_0$ with 
$\phi(z\1) = z^n$ for all $z \in \K$ and 
algebra homomorphisms 
$\chi_1,\ldots, \chi_n \: \cA \to \K$ with 
$\phi = \prod_{i = 1}^n \chi_i.$

To verify this claim, we first consider the polynomial 
$\K \to \K, z \mapsto \phi(z\1)$. Since it is multiplicative and non-zero, 
it maps $\K^\times$ into $\K^\times$, so that it has no zero in $\K^\times$. 
This implies that $\phi(z\1) = z^n$ for some $n \in \N_0$ and all 
$z \in \K$. We conclude that $\phi$ is a homogeneous polynomial 
of degree $n$. 

Following the same line of argument as above, we have 
to extend a homomorphism 
$\psi \: S^n(\cA) \to \K$ to an algebra homomorphism 
$\cA^{\otimes n} \to \K$. In view of Lemma~\ref{lem:idealext}, 
this reduces to the problem to show that 
every maximal ideal $J$ of $\cA^{\otimes n}$ has the property that 
$\cA^{\otimes n}/J \cong \K$. Since $\cA$ is assumed to be finitely 
generated, the same holds for the quotient field 
of $\cA^{\otimes n}$, so that it is a quotient of some polynomial 
ring $\K[x_1,\ldots, x_N]$. Therefore the assertion follows from 
Hilbert's Nullstellensatz. 
\end{rem}

\section{Homogeneous line bundles} \label{sec:4}

Let $\chi: P(\mathcal{A}) \rightarrow \C^{\times}$ 
be a holomorphic character. 
We define the associated holomorphic homogeneous line bundle 
$$ \cL_\chi := (G(\cA) \times \C)/P(\cA) := G(\cA) \times_\chi \C$$ 
and write its elements  as $[g,v]$, which are the orbits for the 
$P(\cA)$-action on $G(\cA) \times \C$ by 
$p.(g,v) := (gp^{-1}, \chi(p)v)$. We identify the space 
of holomorphic sections of $\cL_\chi$ with 
\begin{align*}
&\cO_\chi(G(\cA)) \\
&:= \{ f \in \cO(G(\cA)) \: 
(\forall g \in G(\cA))(\forall p\in P(\cA))\, 
f(gp) = \chi(p)^{-1}f(g) \} 
\end{align*} 
by assigning to $f \in \cO_\chi(G(\cA))$ the section defined by 
$s_f(g P(\cA)) := [g, f(g)]$ (cf.\ \cite{MNS09}). 

Let $\widehat{\mathcal{A}}$ denote the space of all unital continuous 
algebra homomorphisms $\mathcal{A} \rightarrow \C$. We recall that this
is a compact space with respect to the weak-$*$- topology 
on the topological dual space $\mathcal{A}'$, and
that the Gelfand transform
$$ \mathcal{G}: \mathcal{A} \rightarrow  C(\hat{\mathcal{A}}),\quad 
a \mapsto  \hat{a}, \quad \hat{a}(\eta):=\eta(a)$$
is a homomorphism of Banach algebras (cf. \cite[Ch.~11]{Ru91}).

For any $\eta \in \hat{\mathcal{A}}$, we obtain 
a homomorphism of Lie algebras
\begin{eqnarray*}
\varphi_{\eta}: \g(\mathcal{A}) \rightarrow \g, \quad x \otimes a \mapsto \eta(a) x, 
\end{eqnarray*}
and, since $G(\cA)$ is $1$-connected, it integrates to a holomorphic 
homomorphism of complex Banach--Lie groups 
\begin{eqnarray*}
\phi_{\eta}^G: G(\mathcal{A}) \rightarrow G(\C).
\end{eqnarray*}
From  $\varphi_{\eta}(\p(\mathcal{A}))\subseteq \p(\C)$, we derive 
$\phi_{\eta}^G(P(\mathcal{A})) \subseteq P(\C)$ because $P(\mathcal{A})$ 
is connected by definition. 
Since the quotient map $\pi_\cA \: G(\cA) \to X(\cA)$ 
is a holomorphic submersion, $\phi_{\eta}^G$ thus induces a holomorphic map
\begin{eqnarray*}
\phi_{\eta}^X : X(\cA) = G(\mathcal{A})/P(\mathcal{A}) \rightarrow 
X(\C) = G(\C)/P(\C)
\end{eqnarray*}
such that the diagram 
\begin{eqnarray*}
\xymatrix{G(\mathcal{A}) \ar[r]^{\phi_{\eta}^G} \ar@{->>}[d]^{\pi_\cA}& G(\C) 
\ar@{->>}[d]^{\pi_\C}\\
X(\mathcal{A})  \ar[r]^{\phi_{\eta}^X}  & X(\C)}
\end{eqnarray*}
commutes.

\begin{rem} \label{r:charcparabolic}
If $\xi:P(\C) \rightarrow \C$ is a holomorphic character, then by the 
isomorphism $P(\C) \cong U \rtimes L \cong U \rtimes (S \rtimes C)$ 
(cf. Remark \ref{r:cparabolic}), and the
fact that $S$ is connected, semisimple, and $\mathfrak{u} \subseteq [\p(\C), \p(\C)]$,
it follows that $\xi$ is uniquely determined by its restriction 
to the subgroup $C$. Hence the group $\hat P(\C)$ of holomorphic characters 
$P(\C) \rightarrow \C^{\times}$ 
is generated by the characters of the form $\xi_\alpha$, where
$\L(\xi_\alpha)\res_\fh  = \omega_\alpha \in \fh^*$, 
$\alpha \in \Pi_\Sigma$, is the fundamental weight 
with $\omega_\alpha(\check\beta) =\delta_{\alpha,\beta}$ for 
$\beta \in \Pi$. For each $\xi \in \hat P(\C)$ we thus obtain 
$$ \xi = \prod_{\alpha \in \Pi_\Sigma} \xi_\alpha^{\L(\xi)(\check \alpha)}$$ 
and $\xi$ is dominant if and only if $\L(\xi)(\check \alpha) \geq 0$ 
holds for each $\alpha \in \Pi_\Sigma$, which in turn implies 
$\L(\xi)(\check \alpha) \geq 0$ for each $\alpha \in \Delta^+$. 

According to the classical Borel--Weil Theorem, 
in our convention  \eqref{eq:parsyst} for the positive system, 
$\xi$ is dominant 
if and only if the holomorphic line bundle $\cL_\xi$ over 
$X(\C)$ has non-zero holomorphic sections. 
\end{rem}

Assume now that $\chi$ is the pullback of a holomorphic character
$\xi$ of $P(\C)$ with respect to 
$\phi_{\eta}^G|_{P(\mathcal{A})}:P(\mathcal{A}) \rightarrow P(\C)$,
i.e., it is of the form $\chi=\xi \circ \phi_{\eta}^G|_{P(\mathcal{A})}$. 
Then the corresponding line bundle $\cL_{\chi}$ is the pullback
of the line bundle $\mathcal{L}_{\xi}$ over $X(\C)$ with respect
to $\phi_{\eta}^X$. If $\xi$ is dominant, we can
produce holomorphic sections of $\cL_{\chi}$ by pulling back
holomorphic sections to $\mathcal{L}_{\xi}$.
This proves one half of the following theorem. 

\begin{theorem}\label{big}
Let $G(\mathcal{A})$ be a $1$-connected Banach--Lie group 
with Lie algebra $\g(\cA)$, 
$P(\mathcal{A})$ a connected parabolic subgroup of $G(\cA)$, 
and $\chi: P(\mathcal{A}) \rightarrow \C^{\times}$ 
be a holomorphic character.
Then the line bundle $\cL_{\chi}$ over $X(\cA) 
= G(\mathcal{A})/P(\mathcal{A})$
has nonzero global holomorphic sections
if and only if there exist $\eta_1, \ldots, \eta_m 
\in \widehat{\mathcal{A}}$ and fundamental 
holomorphic characters $\xi_1, \ldots, \xi_m$ of $P(\C)$ such that
\begin{eqnarray}
\chi=\Pi_{j=1}^m (\phi_{\eta_j}^G)^*\xi_j. \label{goodchar}
\end{eqnarray}
This implies in particular that $\cL_\chi$ 
is the tensor product of line bundles of the form
$(\phi_{\eta_j}^X)^*\mathcal{L}_{\xi_j}$, where 
$\L(\xi_j)$ is a fundamental weight of $\g$. 
\end{theorem}

\begin{rem} (a) If the group $G(\cA)$ is connected, but not simply connected, 
then the preceding theorem applies to the universal covering 
group $q \: \tilde G(\cA) \to G(\cA)$. 
If $\tilde P(\cA)$, resp., $P(\cA)$, denote the connected subgroups 
of $\tilde G(\cA)$, resp., $G(\cA)$ with Lie algebra $\fp(\cA)$, 
then we derive that $\cO_\chi(G(\cA)) \not=\{0\}$ implies that
the character $\tilde\chi := q^*\chi \: \tilde P(\cA) \to \C^\times$ 
is a product 
\begin{equation}\label{prod2} 
\tilde\chi=\Pi_{j=1}^m (\phi_{\eta_j}^G)^*\xi_j. 
\end{equation}
Since, in general, an algebra homomorphism 
$\eta_j \: \cA \to \C$ does not lead to a group 
homomorphism $G(\cA) \to G(\C)$, we have to face the difficulty 
to express the information directly with respect to the group 
$G(\cA)$. 

(b) However, if $G(\cA)$ is a functorially attached to 
$\cA$, such as the groups 
$\SL_n(\cA)_0$, $\Sp_{2n}(\cA)_0$ or $\SO_n(\cA)_0$, 
every algebra homomorphism $\eta \: \cA \to \C$ induces a 
morphism of Banach--Lie groups $\phi_\eta^G \: G(\cA) \to G(\C)$, 
regardless of whether $G(\cA)$ is simply connected or not. 
Then \eqref{prod2} implies that
$$ q^*\chi = \tilde\chi = q^*\prod_{j=1}^m (\phi_{\eta_j}^G)^*\xi_j, $$
which immediately leads to 
\begin{equation}
  \label{eq:prod3}
\chi =\prod_{j=1}^m (\phi_{\eta_j}^G)^*\xi_j 
\end{equation} 
because $q \: \tilde P(\cA) \to P(\cA)$ is surjective. 
\end{rem}

\begin{rem} \label{rem:findimsub} The preceding theorem implies in particular 
that if $\cL_\chi$ has non-zero holomorphic sections, 
then it is a tensor product of pullbacks of finite dimensional 
line bundles over $X(\C)$. Accordingly, the products of the 
pullbacks of the finite dimensional spaces of holomorphic 
sections of these line bundles over 
$X(\C)$ form a finite dimensional non-zero $G$-invariant subspace of 
$\cO_\chi(G(\cA))$. 
The results in \cite{MNS09} imply that 
this subspace is contained in every closed $G(\cA)$-invariant subspace. 
However, the $G(\cA)$-module structure on the Banach space 
$\cO_\chi(G(\cA))$ is far from being semisimple. 
As we shall see in Section~\ref{sec:7} below, 
the space $\cO_\chi(G(\cA))$ can be infinite dimensional, although 
it contains a finite dimensional minimal non-zero subspace. 
\end{rem}

\section{Proof of Theorem \ref{big}} \label{sec:5}

\subsection{The $\mathfrak{sl}_2$-case}
In this section we consider the special case 
$\g = \fsl_2(\C)$, $\g(\cA)=\mathfrak{sl}_2(\cA)$, where 
$G(\cA):=\tilde\SL_2(\mathcal{A})_0$  is the simply connected covering group 
of the identity component $\SL_2(\cA)_0$ of 
$\SL_2(\cA)$. We consider the parabolic subalgebra $\fp$ 
of upper triangular matrices in $\fsl_2(\C)$, and put 
$$ \fh = \C \check{\alpha}, 
\quad 
\check \alpha = \left(\begin{array}{cc}
1 & 0\\
0 & -1
\end{array}
\right),  \quad \Delta^+ = \{-\alpha\}, 
\quad \mbox{ and } \quad 
\fu = \C \pmat{ 0 & 1 \\ 0 & 0}. $$

\begin{defn} \label{def:hdef} For $z \in \cA^\times$, we define 
$\tilde h(z) \in G(\cA)$ by 
\begin{align*}
\tilde{h}(z)&:= \\
\exp&
\left(\begin{array}{cc}
0 & 0\\
z^{-1}-1 & 0 
\end{array}\right)
\exp
\left(\begin{array}{cc}
0 & 1\\
0 & 0 
\end{array}\right)
\exp
\left(\begin{array}{cc}
0 & 0\\
z-1 & 0 
\end{array}\right)
\exp
\left(\begin{array}{cc}
0 & -z^{-1}\\
0 & 0 
\end{array}\right),
\end{align*} 
and observe that $\tilde h(\1)= \1$.
If $q \: {G(\cA)} \to \SL_2(\cA)_0$ is the universal 
covering map with $\L(q) = \id_{\fsl_2(\cA)}$, then 
\begin{align*}
h(z) &:= q(\tilde h(z)) \\
=& \left(\begin{array}{cc}
1 & 0\\
z^{-1}-1 & 1 
\end{array}\right)
\left(\begin{array}{cc}
1 & 1\\
0 & 1 
\end{array}\right)
\left(\begin{array}{cc}
1 & 0\\
z-1 & 1 
\end{array}\right)
\left(\begin{array}{cc}
1 & -z^{-1}\\
0 & 1 
\end{array}\right)= \left(\begin{array}{cc}
z & 0\\
0 & z^{-1} 
\end{array}\right). 
\end{align*} 
\end{defn}

\begin{lemma} \label{lem:hlift}
For each $a \in \mathcal{A}$, we have 
$\widetilde{h}(\exp_\cA a)= \exp\left( \begin{array}{cc}
a & 0\\
0 & -a 
\end{array}\right),$ where 
$$ \exp_\cA \: \cA \to \cA^\times, \quad x \mapsto e^x = \sum_{n = 0}^\infty 
\frac{x^n}{n!} $$
is the exponential function of the Banach--Lie group $\cA^\times$.  
\end{lemma}

\begin{proof}
From $q(\widetilde{h}(\exp_\cA a))=h(\exp_\cA a)=\exp_{\SL_2(\cA)}\left( \begin{array}{cc}
a & 0\\
0 & -a 
\end{array}\right)$, 
we derive that $\widetilde{h} \circ \exp_\cA: 
\mathcal{A} \rightarrow G(\cA)$ is the
unique continuous lift of the map 
$$\cA\to \SL_2(\cA), \quad 
a \mapsto h(\exp_{\cA}(a)) = \exp_{\SL_2(\cA)}(a\check \alpha)$$ satisfying
$\widetilde{h}(\exp(0))=\boldsymbol{1}$. Since 
$a \mapsto \exp(a \check \alpha)$ is another lift with this 
property, the uniqueness of lifts implies the assertion.
\end{proof}

\begin{proposition}\label{extend2}
Let $\chi: P(\cA) \rightarrow \C^{\times}$ be a holomorphic character 
and observe that it defines a holomorphic character 
\begin{equation} \label{algchar2}
\chi_\cA : (\cA,+) \to \C^\times, \quad 
a \mapsto \chi\big(\exp(a \check \alpha)\big)^{-1}. 
\end{equation}
If the line bundle $\cL_{\chi}$ over ${G(\cA)}/P(\cA)$ admits nonzero 
holomorphic sections, then 
$\chi_\cA$ vanishes on the kernel $\ker(\exp_\cA)$ of the 
exponential function 
$$\exp_\cA \: \cA \to \cA^\times_0,$$ 
and 
induces a holomorphic character $\oline\chi_\cA \: \cA^\times_0 \to\C^\times$ 
which extends to a holomorphic character 
$(\mathcal{A},\cdot) \rightarrow \C$ of 
the multiplicative semigroup $(\mathcal{A},\cdot)$.
\end{proposition}

\begin{proof} We identify the space of holomorphic sections of 
$\cL_\chi$ with the space \break $\mathcal{O}_{\chi}({G(\cA)})$ 
of holomorphic functions $f: {G(\cA)}\rightarrow \C$ which are 
equivariant for $P(\cA)$ 
in the sense that $f(gp)=\chi(p)^{-1}f(g)$ holds for 
$g \in {G(\cA)}$ and $p \in P(\cA)$. 
If this space is nonzero, then \cite[Thm.~3.7]{MNS09} implies 
the existence of an $\fn(\cA)$-invariant function $f \in \cO_\chi({G(\cA)})$ 
with $f(\1) = 1$. This implies that $f$ is $N(\cA)$-left invariant 
and hence in particular 
$f(N(\cA)) = \{1\}.$
Next we note that $\chi$ vanishes on $U(\mathcal{A})$ since 
$\mathfrak{u}(\mathcal{A}) \subseteq  [\fp(\cA), \fp(\cA)]$, so that 
$f$ is also $U(\cA)$-right invariant. Therefore 
$$ f(\tilde h(z)) = 
f\Big(\exp\left(\begin{array}{cc}
0 & 1\\
0 & 0 
\end{array}\right)
\exp
\left(\begin{array}{cc}
0 & 0\\
z-1 & 0 
\end{array}\right)\Big), $$
and the right hand side defines a holomorphic function 
on $\cA$. On the hand, Lemma~\ref{lem:hlift} 
implies that, for $a \in \cA$,  
$$ f(\tilde h(\exp_\cA a)) 
= f(\exp(a \check \alpha)) 
= \chi(\exp(a \check \alpha))^{-1} 
= \chi_\cA(a). $$
First, this proves that 
$\ker \exp_\cA \subeq \ker \chi_\cA$, so that 
$\chi_\cA$ factors through a holomorphic character 
$\oline\chi_\cA \: \cA^\times_0 = \exp_\cA(\cA) \to \C^\times$ 
with $\oline\chi_\cA \circ \exp_\cA = \chi_\cA$. 
For $z \in \cA^\times_0$, we now have 
$\oline\chi_\cA(z) = f(\tilde h(z))$, and we have just 
seen that this function extends to a holomorphic function 
on all of $\cA$. Since this function is multiplicative on all pairs 
in the open subset $\cA^\times_0$, it follows by 
analytic continuation that it is multiplicative. 
\end{proof}

We can now prove Theorem \ref{big} for  $\g=\mathfrak{sl}_2(\C)$.

\begin{proof} (of Theorem \ref{big} for $\g=\mathfrak{sl}_2(\C)$)
If $\cL_\chi$ has  nonzero 
holomorphic sections, then the preceding proposition 
implies the existence of a multiplicative holomorphic function 
$\oline\chi_\cA \: \cA \to \C$ with 
$$ \chi(\exp(a \check \alpha))^{-1} = \oline\chi_\cA(\exp_\cA a) 
\quad \mbox{ for } \quad a \in \cA. $$
Theorem~\ref{holchar} now implies that 
$\oline\chi_\cA = \eta_1 \cdots \eta_n$ for algebra 
homomorphisms \break $\eta_j \: \cA \to \C$, and this implies that 
$$ \chi(\exp(a \check \alpha))^{-1} = \prod_{j = 1}^n \eta_j(\exp_\cA a)
\quad \mbox{ for } \quad a \in \cA. $$
With the dominant fundamental character  (with respect to 
$\Delta^+ = \{-\alpha\}$) 
$$\xi \: P(\C) = \Big\{ 
\pmat{a & b \\ 0 & a^{-1}} \: a \in \C^\times, b \in \C\Big\} \to \C^\times, 
\quad \xi\Big(\pmat{a & b \\ 0 & a^{-1}}\Big) := a^{-1}, $$
 we now obtain
 \begin{align*}
\chi(\exp(a \check \alpha))
&= \prod_{j = 1}^n \eta_j(\exp_\cA a)^{-1}
= \prod_{j = 1}^n e^{-\eta_j(a)} 
= \prod_{j = 1}^n \xi(\exp_{\SL_2(\C)} \eta_j(a)\check \alpha)\\
&= \prod_{j = 1}^n \big((\phi_{\eta_j}^G)^*\xi\big)(\exp(a \check \alpha)).
 \end{align*}
As $P(\cA) = \exp(\cA \check \alpha) U(\cA)$ and $\chi$ vanishes on 
$U(\cA)$, this implies that 
\begin{equation}
  \label{eq:prodsl2}
\chi =  \prod_{j = 1}^n (\phi_{\eta_j}^G)^*\xi.
\end{equation}

The character $\xi$ satisfies $\L(\xi)(-\check \alpha) = 1$, 
hence it is dominant and fundamental. 
The corresponding line bundle $\mathcal{L}_{\xi}$ then admits nonzero
holomorphic sections by the classical Borel--Weil Theorem, 
and this implies that the line bundle 
$\cL_\chi \cong  \otimes_{j = 1}^n (\phi_{\eta_j}^G)^*\cL_\xi$
has non-zero holomorphic sections. 
\end{proof}

\begin{rem} The line bundle $\cL_\xi$ is
the bundle of hyperplane sections over the Riemann sphere. 
Its space of holomorphic sections is the two-dimensional 
fundamental representation of $\SL_2(\C)$ on the dual space 
of $\C^2$. 
\end{rem}

\subsection{The general case}

\begin{proof} (of Theorem \ref{big})
We have already seen that \eqref{goodchar} is sufficient 
for the existence of non-zero holomorphic sections. We now 
prove that it is also necessary.

For any simple root $\alpha \in \Pi$, 
consider the $\mathfrak{sl}_2(\C)$-subalgebra 
$$\g^\alpha := \C \check \alpha + \g_\alpha + \g_{-\alpha} 
\subeq \g \quad \mbox{ and } \quad 
\fp^\alpha := \C \check \alpha + \g_\alpha. $$
For the corresponding $1$-connected Banach--Lie groups, 
we then have morphisms 
$$ \gamma_\alpha^G \: G^\alpha(\cA) \to G(\cA) $$
integrating the inclusion maps $\g^\alpha(\cA) \into \g(\cA)$. 
Clearly, $\gamma_\alpha^G(P^\alpha(\cA)) \subeq P(\cA)$ 
holds for the corresponding connected parabolic subgroups 
$P^\alpha(\cA) \subeq G^\alpha(\cA)$. 

If $\cO_\chi(G(\cA))\not=\{0\}$, then, in view of the left 
invariance of this space, we also have that 
$$\{0\} \not= \gamma_\alpha^*\cO_\chi(G(\cA)) 
\subeq \cO_{\chi^\alpha}(G^\alpha(\cA)), 
\quad \mbox{ where } \quad 
\chi^\alpha := \chi \circ \gamma_\alpha^G\res_{P^\alpha(\cA)}. $$
From the $\fsl_2$-case we thus obtain algebra 
homomorphisms $\eta_1^\alpha,\ldots, \eta_{n_\alpha}^\alpha \in \hat\cA$ 
with 
$$ \chi^\alpha = \prod_{j = 1}^{n_\alpha} (\phi_{\eta_j^\alpha})^*\xi, $$
where $\phi_{\eta_j^\alpha} \: G^\alpha(\cA) \to \SL_2(\C)$ 
is the corresponding evaluation homomorphism. 

If $\iota_\alpha \: G^\alpha(\C) \cong \SL_2(\C) \to G(\C)$ is the 
homomorphism integrating the inclusion $\g^\alpha \to \g$, then 
the corresponding fundamental weight $\xi_\alpha \in \hat P(\C)$ 
satisfies $\iota_\alpha^*\xi_\alpha = \xi$ because 
$\L(\xi)(\check \alpha) =-1$. 
With  the evaluation homomorphisms 
$\phi_{\eta_j^\alpha}^G = \iota_\alpha \circ \phi_{\eta_j^\alpha} 
\: G^\alpha(\cA) \to G(\C)$, we thus obtain 
$$ \chi^\alpha = \prod_{j = 1}^{n_\alpha} 
(\phi_{\eta_j^\alpha}^G)^*\xi_\alpha. $$
Since 
$$\fp(\cA) = [\fp(\cA),\fp(\cA)] \oplus \bigoplus_{\alpha \in \Pi_\Sigma} 
\cA \check \alpha $$ 
(cf.\ Remark~\ref{r:cparabolic}), 
the restrictions to the subgroups $P^\alpha(\cA)$ of $P(\cA)$, 
$\alpha \in \Pi_\Sigma$, determine the character $\chi$ via 
$$ \chi\Big(\prod_{\alpha \in \Pi_\Sigma} \exp(a_\alpha \check \alpha)\Big) 
= \prod_{\alpha \in \Pi_\Sigma} \chi^\alpha(\exp(a_\alpha \check \alpha)) 
= \prod_{\alpha \in \Pi_\Sigma} \prod_{j = 1}^{n_\alpha} 
\chi_j(\exp a_\alpha)^{-1} $$
(cf. Remark \ref{r:charcparabolic}).
We conclude that $\chi$ is a product of pullbacks of 
dominant fundamental characters $\xi_\alpha$ by 
certain evaluation homomorphisms $\phi^G_\chi$, 
and this completes the proof. 
\end{proof}

\section{The case of $C^*$-algebras} \label{sec:6}

Theorem~\ref{holchar} provides a complete classification 
of all homogeneous holomorphic line bundles over 
the spaces $X(\cA)$ with non-zero holomorphic sections. 
The main reason for these bundles playing a role in the 
classical context $\cA = \C$ is that the Borel--Weil Theorem 
asserts that the corresponding spaces  
of holomorphic sections are always irreducible and 
that every irreducible finite dimensional holomorphic 
representation of $G(\C)$ can be realized in this way 
if $P(\C)$ is a Borel subgroup, i.e., $\Pi_\Sigma = \Pi$. 

For general commutative Banach algebras, one cannot expect 
such a sharp picture, as the examples discussed in \cite{NS09} show. 
Here the main source of the lacking semisimplicity of the 
representations lies in the algebra, the simplest examples arising 
for the two-dimensional 
algebra $\cA = \C[\eps]$ of dual numbers, where $\eps^2 = 0$. 
As we know from the Gelfand theory of commutative Banach algebras, 
commutative $C^*$-algebras are the prototype of commutative 
Banach algebras, and any semisimple commutative 
Banach algebra 
$\cA$ embeds continuously into the $C^*$-algebra $C(\hat\cA)$ by 
the Gelfand transform. 

In this section we therefore study the special case 
where $\cA$ is a commutative unital $C^*$-algebra. 
Let $\sigma \: \g \to \g$ be an involutive antilinear automorphism 
satisfying $\sigma(\g_\alpha) = \g_{-\alpha}$ for each 
root $\alpha$ and observe that this implies that $\sigma(\fh) = \fh$ 
and 
$$ \sigma(\check \alpha) = -\check \alpha \quad \mbox{ for } \quad 
\alpha \in \Delta. $$
We combine $\sigma$ with the 
involution $a \mapsto a^*$ of $\cA$ to an antilinear map 
$$ * \:  
\g(\cA) \to \g(\cA), \quad x \otimes a \mapsto -\sigma(x) \otimes a^*,$$
satisfying 
$$ (x^*)^* = x \quad \mbox{ and } \quad 
[x,y]^* = [y^*,x^*]\quad \mbox{ for } \quad x,y\in \g(\cA). $$
Thus $(\g(\cA),*)$ is an {\it involutive Banach--Lie algebra}. 
In view of the Gelfand isomorphism, we have 
$\cA \cong C(\hat\cA)$ and, accordingly, 
$\g(\cA) \cong C(\hat\cA,\g)$ with $f^*(\chi) = -\sigma(f(\chi))$. 
Since $G(\cA)$ was assumed to be simply connected, there exists an 
antiholomorphic involution $g \mapsto g^*$ on $G(\cA)$ satisfying 
$$ (gh)^* = h^*g^* \quad\mbox{ and } \quad 
(\exp x)^* = \exp(x^*)\quad \mbox{ for } \quad 
g,h \in G(\cA), x \in \g(\cA). $$

A {\it holomorphic involutive representation of $(G(\cA),*)$} is a pair 
$(\pi, \cH)$ consisting of a complex Hilbert space $\cH$ 
and a holomorphic homomorphism \break $\pi \: G(\cA) \to \GL(\cH)$ 
which is compatible with the involutions in the sense that  
$$ \pi(g^*) = \pi(g)^* \quad \mbox{ for } \quad g \in G(\cA). $$ 
Such a representation is said to be {\it irreducible} if 
$\cH$ contains no non-trivial $G(\cA)$-invariant closed subspace. 
We write $\dd\pi \: \g(\cA)\to B(\cH)$ for the derived representation 
of $\g(\cA)$ by bounded operators on $\cH$.

Assume that $(\pi,\cH)$ is an irreducible involutive holomorphic 
representation of $G(\cA)$ and 
$P(\cA)$ is a connected parabolic subgroup.  
Then \cite[Thm.~5.1]{MNS09} implies that 
$E := \cH/\oline{\fu(\cA)\cH}$ carries an irreducible holomorphic 
representation $\rho$ of $P(\cA)$ with $U(\cA) \subeq \ker \rho$, 
and the quotient map $\beta \: \cH \to E$ leads to   
an inclusion of holomorphic $G(\cA)$-representations 
$$ \beta_G \: \cH \to \cO_\rho(G(\cA),E), \quad 
\beta_G(v)(g) := \beta(\pi(g)^{-1}v), $$
where 
\begin{align*}
&\cO_\rho(G(\cA),E) \\
&:= \{ f \in \cO(G(\cA),E) \: 
(\forall g \in G(\cA))(\forall p\in P(\cA))\  
f(gp) = \rho(p)^{-1}f(g) \} 
\end{align*}
corresponds to the space of holomorphic sections of the associated 
holomorphic vector bundle $G(\cA) \times_\rho E$ over $X(\cA)$. 
The closure of the image of $\beta_G$ is the unique minimal 
closed $G(\cA)$-invariant subspace of $\cO_\rho(G(\cA),E)$. 

From the construction of the involution $*$ on $\g(\cA)$, we 
immediately derive that $\fu(\cA)^* = \fn(\cA)$ and 
$\fl(\cA)^* = \fl(\cA)$. In particular, the subgroup 
$L(\cA) \subeq G(\cA)$ is $*$-invariant, hence also carries the 
structure of a complex involutive Banach--Lie group. 
Since the subspace $\fu(\cA)\cH$ of $\cH$ is $L(\cA)$-invariant 
and $L(\cA)^* = L(\cA)$, the orthogonal complement 
$E \cong (\fu(\cA)\cH)^\bot$ is also $L(\cA)$-invariant, so that 
the representation $(\rho,E)$ of $L(\cA)$ actually 
is involutive.

If, in addition, $P(\cA)$ is minimal parabolic, 
then $\fl = \fh$ shows that $\fl(\cA)$ is abelian, so that 
Schur's Lemma implies that $\dim E = 1$, and thus 
$\rho \: P(\cA) \to \GL(E) \cong \C^\times$ is a holomorphic 
character and the results developed above apply. 
In particular, Remark~\ref{rem:findimsub} implies that 
the minimal $G(\cA)$-invariant subspace of 
$\cO_\rho(G(\cA),E)$ is finite-dimensional, so that 
$\dim \cH < \infty$. 
Further, the representation of 
$G(\cA)$ on the minimal submodule factors through some 
evaluation homomorphism 
$$ (\phi^G_{\eta_1}, \ldots, \phi^G_{\eta_m}) \: 
G(\cA) \to G(\C^m) \cong G(\C)^m, $$
where $\rho = \prod_{j = 1}^m (\phi_{\eta_j}^G)^*\xi_j$. 
To see how this information can be made compatible with the 
involution, we 
note that 
$\ker \pi \trile \g(\cA)$ is a $*$-invariant ideal. 
Hence it is in particular $\g$-invariant, and since 
$\g(\cA) \cong \g \otimes \cA$ is an isotypical semisimple 
$\g$-module, the closed $\g$-submodule $\ker \pi$ is of the form 
$\ker \pi = \g \otimes I$, where $I \subeq \cA$ is a 
closed $*$-invariant subspace. As $\ker \pi$ is an ideal, 
the relation 
$$ [x \otimes a, \ker \pi] 
= [x,\g] \otimes a I $$
implies that $I \trile \cA$ is an ideal. 
As $\dd\pi(\g(\cA))$ is finite dimensional, 
$\cA/I$ is a finite dimensional $C^*$-algebra, so 
that the Gelfand Representation Theorem implies that 
$\cA/I \cong \C^N$ as $C^*$-algebras. We conclude that 
$$ \pi(\g(\cA)) \cong \g \otimes \C^N \cong \g^N, $$ 
as involutive Lie algebras. Since $(\g,-\sigma)$ has an 
involutive Hilbert representation, the fixed point algebra 
$\g^\sigma$ is compact, so that $\sigma$  actually is a 
Cartan involution and $\g^\sigma$ is a compact real form.

We collect the result of the preceding discussion in the following  
theorem. 

\begin{thm} \label{thm:6} 
Let $(\cA,*)$ be a commutative $C^*$-algebra, 
$\sigma \in \Aut(\g)$ be an involutive automorphism 
with $\sigma(\g_\alpha) = \g_{-\alpha}$ for each $\alpha \in \Delta$ 
and define an involution on $\g(\cA)$ by 
$(x \otimes a)^* := - \sigma(x) \otimes a^*$. 
Let $(G(\cA),*)$ be the $1$-connected involutive Banach--Lie group 
corresponding to $(\g(\cA),*)$. 
Then every irreducible involutive representation 
$(\pi, \cH)$ of $G(\cA)$ is finite dimensional and factors through 
an involutive surjective multi-evaluation homomorphism 
$$ \phi^G_\eta \: G(\cA) \to G(\C)^N, \quad 
g \mapsto (\phi^G_{\eta_1}(g), \ldots, \phi^G_{\eta_N}(g)). $$
\end{thm}

\begin{rem} From the preceding theorem we can now easily derive a 
description of all irreducible involutive 
representations of $G(\cA)$. Since every finite dimensional involutive 
representation is a direct sum of irreducible ones, this implies a 
classification of all finite dimensional ones. 

As every irreducible involutive representation 
factors through an involutive 
representation of some group $G(\C)^N$, the classification 
problem reduces to a description of all irreducible holomorphic 
involutive representations of this group. In view of Weyl's Unitary Trick, 
this is equivalent to the classification of irreducible 
unitary representations of the maximal compact subgroup
$$  \{ g = (g_1,\ldots, g_N) \in G(\C)^N\: 
(\forall j)\, \sigma_G(g_j) = g_j\} \cong K^N, $$
where $\sigma_G$ is the antiholomorphic involution of $G(\C)$ with 
$\L(\sigma_G) = \sigma$ and $K := G(\C)^\sigma$. 
As all representations of this product group are tensor products 
of irreducible representations of the factor groups $K$, 
their classification follows from the Cartan--Weyl classificaton 
in terms of highest weight modules of $\g(\C)$. 
\end{rem}

\begin{rem} (a) The preceding discussion applies in particular 
to the universal covering $G(\cA) = \tilde\SL_n(\cA)_0$ of the group 
$\SL_n(\cA)_0$ for a commutative unital $C^*$-algebra, 
if we take $\g = \fsl_n(\C)$. Then the factorization of the 
representation through some $G(\C)^N = \SL_n(\C)^N$ 
even implies that all irreducible involutive  
representations of $G(\cA)$ factor through $\SL_n(\cA)_0$ because
the evaluation homomorphisms to $\SL_n(\C)^N$ have this property. 

(b) The techniques developed above apply to groups of the 
form $G(\cA)$, where $\g(\cA) = \g \otimes \cA$ and 
$\g$ is semisimple. If we want to extend the results 
to reductive Lie algebras $\g$, 
we observe that, in this case,  
$$ \g(\cA) = (\z(\g)\otimes \cA) \oplus \g'(\cA), $$
where $\g' = [\g,\g]$ is the commutator algebra and 
$\z(\g) \otimes \cA = \z(\g(\cA))$ is central. 

In view of Schur's Lemma, all irreducible 
involutive holomorphic representations 
$(\pi,\cH)$ of 
$$ G(\cA) \cong Z(G(\cA))_0 \times G'(\cA) $$ 
have the property that 
$\pi(Z(G(\cA))) \subeq \C^\times \1$, so that 
$\pi\res_{G'(\cA)}$ is also irreducible. Therefore the 
classification in the reductive case splits into 
the classification in the semisimple case and the 
description of the holomorphic characters of 
$$ Z(G(\cA))_0 \cong \z(\g) \otimes \cA = \z(\g(\cA)), $$
which can be identified with the elements in the dual space 
$\z(\g)^* \otimes \cA'$, which are invariant under the canonical 
extension of~$*$. 

For $\g = \gl_n(\C)$ we have $\z(\g) = \C$, so that 
the holomorphic characters of $Z(G(\cA))_0$ correspond to 
arbitrary hermitian functionals $\alpha \: \cA \to \C$. 
In particular, these functionals do not have to 
factor through evaluation maps. A typical example is 
the Riemann integral $I(f) = \int_0^1 f(x)\, dx$ 
on the commutative $C^*$-algebra $\cA = C([0,1])$. 
\end{rem}

\section{An infinite dimensional space of sections} \label{sec:7}

We have seen in Theorem~\ref{big} that, whenever the space 
$\cO_\chi(G(\cA))$ is nonzero for a holomorphic characters 
$\chi \: P(\cA) \to \C^\times$, then the corresponding line bundle 
$\cL_\chi \to X(\cA)$ is a product of pullbacks of line bundles  
$\cL_\xi \to X(\C)$. Since the corresponding space $\cO_\xi(G(\C))$ 
of holomorphic sections is finite dimensional, we obtain a 
$G(\cA)$-invariant non-zero finite dimensional subspace of 
$\cO_\chi(G(\cA))$. In this section we describe an example 
where the space $\cO_\chi(G(\cA))$ is infinite dimensional. 

Throughout this section, we fix a commutative unital 
Banach algebra $\cA$ and a unital algebra homomorphism 
$\eta \: \cA \to \C$. Let $\cI := \ker \eta$. 
As we shall see below, an important ingredient in our construction 
is the space $T_\eta(\cA) := \cI/\oline{\cI^2}$, and we shall assume that 
this space is infinite dimensional. 

\begin{exs} (a) The simplest examples where $T_\eta(\cA)$ 
is infinite dimensional arises as follows. For a Banach space 
$E$, we consider the unital Banach algebra 
$$ \cA = E \oplus \C \quad \mbox{ with } \quad (v,\lambda)(w,\mu) := 
(\lambda w + \mu v, \lambda\mu), $$
and the homomorphism $\eta \: \cA \to \C, \eta(v,\lambda) := \lambda$. 
Then $\cI = E$ and $\cI^2 = \{0\}$, so that 
$T_\eta(\cA) \cong \cI$ is infinite dimensional if $E$ has this property. 

(b) An example which reminds more of an algebra of functions can be 
constructed as follows. 
Let $E$ be a Banach space and 
$\cA$ be the algebra of all analytic functions 
$f = \sum_{n = 0}^\infty f_n$ ($f_n$  homogeneous of degree $n$) with 
$$ \|f\| := \sum_{n = 0}^\infty \|f_n\| < \infty, $$
where $\|f_n\| := \sup_{\|v\|\leq 1} \|f_n(v)\|$, so that $\cA$ 
can be conidered as an algebra of functions on the closed unit ball 
of $E$. It is easy 
to verify that $\cA$ is a Banach algebra with respect to pointwise 
multiplication. Further $\eta(f) := f(0) = f_0$ 
defines a continuous homomorphism  to $\C$ and 
$$ \cI = \ker \eta = \{ f \in \cA \: f_0 = 0\} $$
implies that 
$$ \cI^2 \subeq \{ f \in \cA \: f_0 = f_1 = 0\}. $$ 
In particular, we obtain an injection 
$E' \into T_\eta(\cA)$, so that this space is infinite dimensional 
if $E$ is. 

(c) Another class of examples can be produced by considering for a unital 
commutative Banach algebra $\cB$ the Banach algebra 
$$ \cA := \Big\{ f = \sum_{n = 0}^\infty t^n f_n \in \cB[[t]] \: 
\sum_{n = 0}^\infty \|f_n\| < \infty, f_0 \in \C \1 \Big\} $$
of formal $\cB$-valued power series converging absolutely for 
$|t| \leq 1$. Then $\eta(f) := f(0) \in \C$ defines a 
homomorphism $\eta \: \cA \to \C$, for which 
$$ \cI = \Big\{ f \in \cA \: f = \sum_{n = 1}^\infty t^n f_n \Big\},
\quad  \cI^2 = \Big\{ f \in \cA \: f = \sum_{n = 2}^\infty t^n f_n \Big\}, $$ 
so that $T_\eta(\cA) \cong \cB$, as a Banach space. 
\end{exs}

\begin{defn} Let $E$ be a Banach space. 
A continuous linear functional $\delta \: \cA \to E$ is 
called an {\it $\eta$-derivation} if 
\begin{equation}
  \label{eq:etader}
 \delta(ab) = \eta(a) \delta(b) + \eta(b) \delta(a) \quad \mbox{ for } 
\quad a,b \in \cA. 
\end{equation}
We write $\Der_\eta(\cA,E)$ for the space of all $\eta$-derivations on~$\cA$. 
\end{defn}

\begin{rem}
Clearly, for $\delta \in \Der_\eta(\cA,E)$, 
the relation $\delta(\1) = 2 \delta(\1)$ leads to 
$\delta(\1) =0$, so that $\delta$ 
is determined by its values on the hyperplane ideal 
$\cI$ and \eqref{eq:etader} further implies that 
$\delta(\cI^2) = \{0\}$. Using the relation 
$$ ab - \eta(ab)\1 
\in \eta(a)(b- \eta(b)\1) + \eta(b) (a - \eta(a)\1) + \cI^2, $$
it is easy to see that, conversely, every continuous 
linear map $\alpha \in \cI \to E$ vanishing on $\cI^2$ 
defines an $\eta$-derivation via 
$$ \delta(a) := \alpha( a - \eta(a)\1). $$ 
This leads to an isomorphism of Banach spaces 
$$ \Der_\eta(\cA,E) \cong B(T_\eta(\cA), E) \cong 
\{ \alpha \in B(\cI,E)  \: \cI^2 \subeq \ker \alpha \}. $$
In this sense the map 
$$ \delta_u \: \cA \to T_\eta(\cA), \quad 
a \mapsto [a - \eta(a)\1] := (a - \eta(a)\1) + \oline{\cI^2} $$
is the universal $\eta$-derivation; every other $\eta$-derivation 
factors uniquely through~$\delta_u$. 
\end{rem}

For the construction of examples where 
$\cO_\chi(G(\cA))$ is infinite dimensional, we shall focus on the  
case where $\g = \fsl_2(\C)$ and $\chi$ is the square of a 
pullback character, i.e., $m = 2$ in the notation of 
Theorem~\ref{big}, resp., \eqref{eq:prodsl2}: 
$$ \chi = (\phi_\eta^G)^*\xi^2. $$

In this case the representation of $\SL_2(\C)$ in the 
space $\cO_{\xi^2}(\SL_2(\C))$ is equivalent to the adjoint 
representation of $\fsl_2(\C)$, so that the pullback by 
$\phi_\eta^G$ leads to an injection 
$$ \fsl_2(\C) \into \cO_\chi(G(\cA)), $$
where $G(\cA)$, resp., the quotient group $\SL_2(\cA)$ acts on 
this space by 
$$ g.x := \Ad(\phi_\eta^G(g))x. $$
Accordingly, we write $\fsl_2(\C)_\eta$ for the $G(\cA)$-module 
$\fsl_2(\C)$, endowed with this action. Using 
\cite[Prop.~2.13]{MNS09}, we see that 
the quotient module \break $\cO_\chi(G(\cA))/\fsl_2(\C)_\eta$ is trivial 
because the only eigenvalue of $\check\alpha$ on this space is zero.
We take this as a motivation to study $\g(\cA)$-modules $V$ which 
are extensions of a trivial module $W$ by $\fsl_2(\C)_\eta$: 
$$ \0 \to \fsl_2(\C)_\eta \into V \onto W \to \0. $$

As $\fsl_2(\C)$ is finite dimensional, the Hahn--Banach Theorem implies 
that any such module can be written as a direct sum 
$V = \fsl_2(\C)_\eta \oplus W$ of Banach spaces, and the 
action is given by 
$$ (x \otimes a).(y,w) = (\eta(a)[x,y] + f(w)(x \otimes a), 0), $$
where 
$$ f \: W \to Z^1(\fsl_2(\cA), \fsl_2(\C)_\eta) $$
is a continuous linear map with values in the Banach space 
of $1$-cocycles on $\fsl_2(\cA)$ with values in $\fsl_2(\C)$. 
We therefore have to analyze the space  
\break $Z^1(\fsl_2(\cA), \fsl_2(\C)_\eta)$. 
Averaging over the compact group corresponding to the subalgebra 
$\su_2(\C) \subeq \fsl_2(\C) \subeq \fsl_2(\cA)$, 
it follows that 
every cocycle \break 
$\beta \in Z^1(\fsl_2(\cA), \fsl_2(\C)_\eta)$ 
is cohomologous to an $\su_2(\C)$-equivariant one, and by 
complex linearity, it is even $\fsl_2(\C)$-equivariant. 
Since $\fsl_2(\C)$ is a simple $\fsl_2(\C)$-module, 
the description of $\fsl_2(\cA)$ as $\fsl_2(\C) \otimes \cA$ 
exhibits $\cA$ as a multiplicity space, and we conclude that 
any $\fsl_2(\C)$-equivariant map $\fsl_2(\cA) \to \fsl_2(\C)$ 
is of the form 
$$ \beta(x \otimes a) = d(a)x $$
for a continuous linear map $d \: \cA \to \C$. 
Now 
$$ \beta([x \otimes a, y \otimes b]) 
=  \beta([x,y] \otimes ab) = d(ab) [x,y] $$
and 
$$ (x \otimes a).\beta(y \otimes b) 
= (x \otimes a) d(b)y = \eta(a) d(b) [x,y] $$
imply that $\beta$ is a $1$-cocycle if and only if $d$ is an 
$\eta$-derivation. This shows that 
$$ Z^1(\fsl_2(\cA), \fsl_2(\C)_\eta)^{\fsl_2(\C)} \cong \Der_\eta(\cA,\C).$$

We therefore consider the Banach space  
$$ V := \fsl_2(\C)_\eta \oplus T_\eta(\cA)' 
\cong \fsl_2(\C)_\eta \oplus \Der_\eta(\cA,\C) $$ 
and observe that 
$$ (x \otimes a).(y, \alpha) := 
(\eta(a)[x,y] + \alpha(a)x, 0) $$
defines a continuous action of $\fsl_2(\cA)$ on $V$ because 
$f(\alpha)(x \otimes a) := \alpha(a)x$ 
defines a continuous linear map 
$$ f \: \Der_\eta(\cA,\C) \to Z^1(\fsl_2(\cA), \fsl_2(\C)). $$

Since the group $G(\cA) = \tilde\SL_2(\cA)_0$ is simply connected, 
this $\fsl_2(\cA)$-module structure integrates to a holomorphic 
representation $(\pi, V)$ of $G(\cA)$. It remains to show that 
$V$ injects into the space $\cO_\chi(G(\cA))$ for 
$\chi = (\phi^G_\eta)^*(\xi^2)$. 

In the following we recall the basis 
$$ h := \pmat{1& 0 \\ 0 & -1}, \quad 
e := \pmat{0& 1 \\ 0 & 0} \quad \mbox{ and } \quad 
f := \pmat{0& 0 \\ 1 & 0} $$ 
of $\fsl_2(\C)$ with the relations 
$$ [e,f] = h, \quad [h,e] = 2e \quad \mbox{ and } \quad [h,f] = -2f. $$
We write $(h^*,e^*,f^*)$ for the corresponding dual basis of 
$\fsl_2(\C)^*$ and recall the fundamental character 
$$ \xi \: P(\C) \to \C^\times, \quad 
\xi\Big(\pmat{a & b \\ 0 & a^{-1}}\Big) := a^{-1}. $$

\begin{lem}\label{lem:emb} 
For $\chi = (\phi^G_\eta)^*(\xi^2)$, the continuous 
linear functional $\mu \in V'$, defined by 
$$ \mu(x,\alpha) := f^*(x) $$ 
defines an embedding 
$$ \mu_G \: V \to \cO_\chi(G), \quad 
\mu_G(v)(g) := \mu(g^{-1}.v) $$
of holomorphic Banach $G(\cA)$-modules. 
\end{lem}

\begin{prf} In view of \cite[Thm.~A.6]{MNS09}, we have 
to show that $\mu$ is $P(\cA)$ equivariant if the action on 
$\C$ is defined by $\chi$, and that $\mu$ is $G(\cA)$-cyclic, i.e., 
that $\mu(G(\cA).v) = \{0\}$ implies $v =0$. 

First we verify the equivariance. 
For $x \otimes a \in \fp(\cA)$ with $x = \gamma h + \beta e$, we have 
\begin{align*}
\mu((x \otimes a).(y,\alpha)) 
&= \mu(\eta(a)[x,y] + \alpha(a)x,0)
= \eta(a)f^*([x,y]) 
= \gamma\eta(a)f^*([h,y]) \\
&= -2\gamma \eta(a) f^*(y) 
= -2\gamma \eta(a) \mu(y,\alpha)
= 2 \L(\xi)(x)\eta(a) \mu(y,\alpha) \\ 
&= \L(\chi)(x \otimes a)\mu(y,\alpha), 
\end{align*}
so that the connectedness of $P(\cA)$ implies the equivariance 
of $\mu$. 

Next we show that $\mu$ is cyclic. 
Let $v = (y, \alpha) \in V$ be such that 
$\mu(G(\cA).v) =\{0\}$. By taking first order derivatives, 
we obtain $\mu(\fsl_2(\cA).v) = \{0\}$, which 
implies in particular that, for each $a \in \cA$, we have  
\begin{align*}
0 &= \mu(f \otimes a.(y,\alpha)) 
= \mu(\eta(a)[f,y] + \alpha(a)f,0)\\
&= \eta(a)f^*([f,y]) + \alpha(a) f^*(f)
= \eta(a)f^*([f,y]) + \alpha(a). 
\end{align*}
For $a =\1$ this leads to $f^*([f,y]) = 0$, so that 
$y \in \Spann \{f,e\}$. For $a \in \cI$ we further obtain 
$\alpha = 0$. Now $0= \mu(y,\alpha) =f^*(y)$ implies that 
$y = \lambda e$ for some $\lambda \in \C$. 
We also have 
\begin{align*}
0&= \mu((f \otimes a)^2.(y,0)) 
= \eta(a)\mu(\eta(a)[f,[f,y]],0)
= \eta(a)^2 f^*([f,[f,y]]) \\
&= \lambda \eta(a)^2 f^*([f,-h]) 
= (-2)\lambda \eta(a)^2, 
\end{align*}
which leads to $\lambda = 0$, and hence to $(y,\alpha) =0$. 
This proves that $\mu$ is cyclic. 
\end{prf}

The preceding lemma implies in particular that 
$\cO_\chi(G(\cA))$ is infinite dimensional if 
$T_\eta(\cA)$ has this property. This completes the proof of: 

\begin{prop} Let $\cA$ be a unital commutative Banach algebra 
and $\chi \in \hat\cA$ with kernel $\cI$ such that 
$T_\chi(\cA) = \cI/\oline{\cI^2}$ 
is infinite dimensional. 
Let $G := \tilde\SL_2(\cA)_0$ and 
$\phi_G^\eta \: G\to \SL_2(\C)$ denote the group 
homomorphism induced by $\eta$. Then, for the holomorphic character 
$$\xi\Big(\pmat{a & b \\ 0 & a^{-1}}\Big) := a^{-1}$$ 
of $P(\C),$ 
and the pullback of its square 
$\chi := (\phi_\eta^G)^*(\xi^2) \: P(\cA) \to \C^\times,$
the Banach space $\cO_\chi(G)$ is infinite dimensional. 
\end{prop}

\end{document}